\documentclass[12pt]{amsart}
\usepackage{amssymb}
\usepackage{amsfonts}
\usepackage{graphicx}
\usepackage{epsfig}

\theoremstyle{plain}
\newtheorem{theorem}{Theorem}[section]
\newtheorem{proposition}[theorem]{Proposition}
\newtheorem{lemma}[theorem]{Lemma}
\newtheorem{corollary}[theorem]{Corollary}
\theoremstyle{definition}

\newenvironment{thA}{\begin{trivlist}\item[]{\bf Theorem  A\ }\begin{em}}
{\end{em}\end{trivlist}}
\newenvironment{thB}{\begin{trivlist}\item[]{\bf Theorem  B\ }\begin{em}}
{\end{em}\end{trivlist}}

\newcommand{\R}{\mathbb{R}}
\newcommand{\bc}{\mathbb{C}}
\newcommand{\ra}{{\rightarrow}}

\let\cal\mathcal

\def\tc{\overline{\mathcal{T}(S)}}

\def\g{\gamma}
\def\s{\sigma}
\def\a{\alpha}
\def\b{\beta}
\def\e{\epsilon}
\def\teich{\mathcal{T}(S)}
\def\H{\mathbb{H}}
\def\ray{\cal{R}}

\begin{document}
\title
        {Asymptotic behavior of grafting rays }
        \author{Raquel D\'{\i}az and Inkang Kim}
        \date{}
        \maketitle

\begin{abstract}
In this paper we study the convergence behavior of grafting rays to
the Thurston boundary of Teichm\"uller space. When the grafting is
done along a weighted system of simple closed curves  or along a
maximal uniquely ergodic lamination this behavior is the same as for
Teichm\"uller geodesics and lines of minima. We also show that the
ray grafted along a weighted system of simple closed curves is at
bounded distance from Teichm\"uller geodesic.
\end{abstract}
 \footnotetext[2]{The first author
gratefully acknowledges the partial support of MCyT, DGI
 (BFM2003-03971) and hospitality of Seoul National
 University.} \footnotetext[3]{The second author
gratefully acknowledges the partial support of KRF grant
(KRF-2006-312-C00044) and the warm support of Universidad
Complutense at Madrid-Grupo Santander.}

\section{Introduction}

A complex projective structure on a surface is a
maximal atlas with charts modelled on $\bc P^1$ whose
transition maps are restrictions of automorphisms of
$\bc P^1$. The space $P(S)$ of projective structures
 on an oriented surface $S$ can be parametrized by the bundle
of  quadratic differentials over the Teichm\"uller space $\teich$,
through the Schwarzian derivative of the developing map of the
projective structure. Another more geometric parametrization of
$P(S)$ was given by Thurston using grafting. Indeed $P(S)$ is
homeomorphic to $\teich \times \cal{ML}(S)$ where $\cal{ML}(S)$ is a
measured lamination space \cite{KT}.
His grafting technique bridges together Teichm\"uller theory,
projective structure and Kleinian group theory  in geometric point
of view. This connection appears between the boundary of the
convex core of a hyperbolic 3-manifold which records a hyperbolic
metric $\sigma$ together with a bending measured lamination
$\lambda$, and its corresponding ideal boundary admitting a
projective structure which is the result of the grafting of
$\sigma$ along $\lambda$. After this technique was introduced,
there has been a tremendous study over this subject. See \cite{Du,
Ka, mcm, Ta} for example for excellent expositions. The difficult
part is to know the hyperbolic metric underlying the projective
structure obtained by grafting. In this paper, we estimate the
hyperbolic lengths of simple closed curves in a grafted structure
and draw some useful information.

On the other hand, Teichm\"uller space has been studied more than
several decades in relation to Teichm\"uller geodesics, dynamics and
related combinatorial structures. The complication lies often in the
thin part of it. Analytic tools like quadratic differentials,
Beltrami differentials have been studied by Gardiner, Strebel and
many others \cite{Ga,St}. Also the relation between quadratic
differentials and measured laminations have been studied by
Gardiner, Masur and Kerckhoff \cite{GM,K}. More recently
Techm\"uller space has been studied by using harmonic maps by
Minsky, Scannell and Wolf \cite{Min1,Wolf}.

Teichm\"uller space $\teich$ has some common aspects with
negatively curves spaces. 
Nevertheless, Masur \cite{Masur2} proved that Teichm\"uller space
does not have negative curvature. His method consists on showing the
existence of two geodesic rays starting at the same point and
remaining at bounded distance from each other. Also by studying the
limit points of some Teichm\"uller geodesics, Masur proved in
\cite{Masur} that the Teichm\"uller  and Thurston boundaries of
Teichm\"uller space are almost  the same. Indeed Kerckhoff had
proved that they are different
 \cite{K}.
 There are also
examples of Teichm\"uller geodesics which do not converge in
Thurston compactification \cite{L}.

There is an another object in Teichm\"uller space, lines of minima,
which were introduced by Kerckhoff \cite{K1}, with a more
hyperbolic-geometric meaning (along these lines some function in
terms of  hyperbolic lengths is minimized, while along Teichm\"uller
geodesics, some function in terms of extremal lengths is minimized).
The convergence properties of these lines towards Thurston boundary
has been studied by D\'{\i}az-Series \cite{DS}, and their similar
asymptotic behavior with Teichm\"uller geodesics by Choi-Rafi-Series
\cite{crs}.

In this paper, we study the third object, grafting rays and their
convergence properties, and prove  analogous results as in
\cite{Masur} for Teichm\"uller geodesics and in \cite{DS} for
lines of minima.
More precisely we prove that
\begin{thA} {\rm (Theorems  \ref{con} and
\ref{limit})}
 For a fixed hyperbolic surface $X$, if either $\lambda$
is a maximal uniquely ergodic measured lamination or a
weighted system $\sum c_i\g_i$ of simple closed
curves, then $gr_{t\lambda}(X)$ converges to
$[\lambda]$ or to $[\sum \g_i]$ respectively in
Thurston boundary.
\end{thA}

A natural question is then whether grafting rays are
quasi-geodesics with respect to the Teichm\"uller metric.

\begin{thB} {\rm (Theorem \ref{quasigeodesic})}
 If $\lambda=\sum c_i\g_i$ is a weighted
system of simple curves, then the grafting ray
$gr_{{t}\lambda}(X)$ is within a bounded distance from a geodesic
ray in $\cal T(S)$.
\end{thB}

The proofs follow the similar lines as in \cite{crs, Masur,DS}. To
prove the convergence properties, we need to estimate the length
of simple closed curves along the ray. We use the general
estimates of the lengths of curves given by Proposition
\ref{brokenarc}. Then we analyze in Propositions \ref{lengthbound}
and \ref{twists} how the different terms in this estimate behave
along the grafting rays.

To prove that grafting rays are quasi-geodesics, we estimate the
distance of these rays with   Teichm\"uller geodesics by using
Minsky's product region theorem \cite{Min}.

\section{Preliminaries}

\subsection{  Teichm\"uller space, metric and geodesics.}
Let $S$ be a smooth orientable surface with finitely
many punctures, of hyperbolic type. The Teichm\"uller
space of $S$ is the set $\mathcal {T}(S)$ of conformal
structures on $S$ up to conformal homeomorphisms
homotopic to the identity. Each conformal structure
determines a class of metrics compatible with the
conformal structure. By the uniformization theorem,
among these metrics there is a unique Riemannian
hyperbolic metric, and conversely, each hyperbolic
structure $\sigma$ on $S$ has an underlying conformal
structure. Thus, we can also consider $\teich$ as the
set of hyperbolic metrics on $S$, where two metrics
$\sigma, \sigma'$ are equivalent if there is an
isometry $(S,\sigma)\to (S,\sigma')$ isotopic to the
identity. As notation, when we have a curve $\g$ on a
surface $S$ and we consider a hyperbolic structure
$\s$ on $S$, then we will denote by $\ell_{\s}(\g)$
the hyperbolic length of the geodesic represenative of
$\g$ on $(S,\s)$. By abuse of notation, most times we
will also denote by $\g$ this geodesic representative.

Given two different points $X,X'\in \teich$, Teichm\"uller showed
that there is a unique quasiconformal homeomorphism $f_0$ between
them, homotopic to the identity,  with the smallest possible
maximal dilatation (or quasiconformal constant) $K[f_0] $. The
{\it Teichm\"uller distance} between $X,X'$ is then defined to be
$d_{\teich}(X,X')= \frac{1}{2} \log K[f_0] $.

To describe the extremal map $f_0$  Teichm\"uller made use of
quadratic differentials. A  meromorphic quadratic differential
$\varphi$ on a Riemann surface $X\in \teich$ has in local
coordinates the expression $ \varphi(z)dz^2$, where $\varphi(z)$
is meromorphic with simple poles  at most in the punctures of $S$.
At points $z_0\in S$ that are not zeros or poles of $\varphi$, we
can consider another coordinate chart $w(z)=\int_{z_0}^z
\sqrt{\varphi(z)}dz$, that determines two well defined foliations
on $S$:   the {\it horizontal foliation} along whose leaves ${\rm
Im} w$ is constant, and the {\it vertical foliation} along whose
leaves ${\rm Re} w$ is constant. There is also a transverse
measure defined on these foliations, making them two measured
foliations $\nu^+, \nu^-$: the {\it horizontal measure} of an arc
$\a=z(t)$ transverse to the horizontal foliation is defined as the
integral $\int_{\a} |Im\sqrt{ (\varphi(z(t))}|dt $, and similarly
for the vertical measure. The space $Q(X)$ of quadratic
differentials on $X$ is a complex vector space of complex
dimension $3g-3+p$, with $g$ the genus of $S$ and $p$   the number
of punctures.

We consider the norm given by $\int_S|\varphi|dxdy$, and let
$B^1(X)$ be the open unit ball and $\Sigma^1(X)$ the unit sphere.
Given a quadratic differential $\varphi\in \Sigma^1$ and a  number
$0\leq k<1$, the {\it Teichm\"uller deformation} of $X$ determined
by $\varphi$ and $k$ is the new complex structure $(X,\varphi,k)$
on $S$ defined by the charts $w'(w) = K^{1/2} {\rm Re}w + i
K^{-1/2} {\rm Im}w$, where $K=\frac{1+k}{1-k}$; that is, we expand
the horizontal lamination of $\varphi$ by a factor of $K^{1/2}$
and contract the vertical lamination by a factor of $K^{-1/2}$.
Notice that the construction gives also a quadratic differential
$\varphi'$ on $(X,\varphi,k)$, with norm 1, whose horizontal and
vertical measured foliations are $K^{-1/2}\nu^+$ and
$K^{1/2}\nu^-$, respectively. The identity on the topological
surface $S$ is the extremal map between $X$ and $(X,\varphi,k)$,
with maximal dilatation $K$.

Teichm\"uller existence and uniqueness theorems imply
that the map $\varphi\mapsto
(X,\frac{\varphi}{\|{\varphi}\|}, \|{\varphi}\|)$ is a
homeomorphism between $B^1(X)$ and $\teich$.

{\bf Teichm\"uller geodesics.}  Given a point $X\in \teich$, and a
measured foliation, it is proven in \cite{HM} that there is a
unique quadratic differential $\varphi\in \Sigma^1(X)$  whose
vertical foliation $\nu^-$ is projectively equivalent to the given
one. Then, the map $t\mapsto (X,\varphi,\frac{t}{t+2})$ with
$t\geq 0$ parameterizes a Teichm\"uller geodesic ray starting at
$X$. We use this parametrization for later purpose. We denote this
geodesic ray by ${\cal G}(\nu^-, X)$, and by ${\cal G}_t(\nu^-,X)$
the point on this geodesic that is the image of $t$.


\subsection{ Thurston's boundary of Teichm\"uller space.}
Let $\cal{ ML}(S)$ denote the space of measured laminations of $S$.
The hyperbolic length of closed geodesics extends by linearity and
continuity to the hyperbolic length of measured laminations.
Equally, the intersection number between closed geodesics extends to
the intersection number of measured laminations.  Thurston proved
that Teichm\"uller space can be compactified by the space $P\cal{
ML}$ of projective measured laminations, so that
$\overline{\teich}=\teich\cup P\cal{ ML}$ is homeomorphic to a
closed ball. Precisely, a sequence of marked hyperbolic surfaces
$(S,\s_n)$ converges to a projective measured lamination $[\nu]$ if
there is a sequence $c_n\to 0$ such that for any other measured
lamination $\a$,
$$
\lim_{n\to \infty}c_n\ell_{\s_n}(\a)=i(\a,\nu).
$$
As a consequence of the definition, if $(S,\sigma_n)\to [\nu]$, then
the length of any lamination $\beta$ satisfying $i(\beta,\nu)\not=0$
tends to infinity along the sequence.

In this paper, we will freely use the natural identification
between measured laminations and measured foliations.

\subsection{Twisting numbers and Fenchel-Nielsen
coordinates}\label{twisting}  Followig Minsky
\cite{Min} (see also \cite{crs}),
 we define the twisting number of a closed geodesic around another curve in a
 hyperbolic surface. Let $\sigma$ be a hyperbolic metric on a surface $S$
and $\gamma$ a homotopically nontrivial oriented
simple closed curve. Suppose a simple closed curve
$\alpha$ intersects $\gamma$ and geodesic
representatives $\gamma^\sigma$ and $\alpha^\sigma$
intersects at $x$. Take lifts of them to $\H^2$ and
denote them by $L_\gamma, L_\alpha$ respectively, so
that they intersect at a point which projects to $x$.
$L_\gamma$ has an orientation because $\g$ is oriented. Let $a_r,a_l$ be
endpoints of $L_\alpha$ to the right and left of
$L_\gamma$. Let $p\colon\H^2\ra L_\gamma$ be the orthogonal
projection, and define the signed {\it twisting
number} of $\alpha$ around $\g$ as
$$tw_{\sigma}(\alpha,\gamma)=\min_{\gamma^\sigma\cap \alpha^\sigma}
 \frac{p(a_r)-p(a_l)}{\ell_\sigma(\gamma)}.$$
 We will denote by $Tw_{\sigma}(\alpha,\gamma)$ the
 absolute value of $tw_{\sigma}(\alpha,\gamma)$.

We remark that one can also  define the twisting number of a
measured lamination around a $\g$ (in fact this number only
depends on the support of the measured lamination).

%
%

\bigskip
Fenchel-Nielsen coordinates are global coordinates for Teichm\"uller
space associated to a fixed {\it marking} of the surface. We recall
its definition here.
 Let $\{\gamma_1,\cdots,\gamma_N\}$ denote a system of simple, oriented,
 closed curves that decompose $S$ into a union of  pairs of pants.  A pair of pants $P$
  with boundary curves
$\gamma_1,\gamma_2,\gamma_3$ contains three unique homotopy classes of
 simple arcs
$\alpha_{12},\alpha_{23},\alpha_{13}$, called seams,
such that $\alpha_{ij}$ joins $\gamma_i$ to
$\gamma_j$. Fix a set of representatives of the seams
which match on opposite sides of $\gamma_i$; this
will determine a system of curves $\mu$, and we call the oriented curves $\gamma_i$
together with $\mu$  a {\it marking} of $S$.  Note that these are topological data.

Now, given a hyperbolic metric $\sigma$ on $S$, each
$\g_i$ and each seam have a unique
 geodesic representative $\g_i^{\sigma}, \alpha_{ij}^{\sigma}$, and $\mu$ has a
 unique representative going along the $ \alpha_{ij}^{\sigma}$ and some arcs
 along $\g_i^{\sigma} $. For each $j=1,\dots, N$, let
 $m_j(\sigma)$ be the signed length of the arc on $\g_j^{\sigma}$ mentioned above
 (the sign is given according to the orientation of $\g_j$).
 The {\it twist parameter}
$t_j(\sigma)$ is defined to be
$\frac{m_j(\sigma)}{\ell_\sigma(\gamma_j)}.$

 The {\it
Fenchel-Nielsen} coordinates for $\sigma$ with respect to the
marking $\{\gamma_1,\cdots,\gamma_N;\mu\}$ are defined to be
$$(\ell_\sigma(\gamma_1),\cdots,\ell_\sigma(\gamma_N),t_1
(\sigma),
\cdots,t_N(\sigma)).$$

{\bf Remark.} Minsky proved in   \cite{Min} that the
twist parameter $t_j(\sigma)$ is almost the same as
the twisting number of $\mu$ around $\g_j$; precisely,
$$|t_j(\sigma)-tw_{\sigma}(\mu,\gamma_j)|
\leq 1.$$ He also showed that  for any two curves $\a,\beta$
intersecting $\g_j$, the difference $|tw_{\sigma}(\a,\gamma_j )-
tw_{\sigma}(\beta,\gamma_j )|$ is independent of $\sigma$ up to a
bounded error (of 1). As a consequence, along a family of
hyperbolic surfaces, bounding the quantity
$t_j(\sigma)\ell_{\sigma}(\g_j)$ is equivalent to bounding
$tw_{\sigma}(\a,\gamma_j)\ell_{\sigma}(\g_j)$ for any curve $\a$
intersecting $\g_j$ if $\ell_{\sigma}(\g_j)$ is bounded.

\subsection{Estimate of the length of a curve}

In this paper we will always be dealing with
hyperbolic surfaces with the property that one
specific pants decomposition system of curves
$\{\g_i\}$ have length bounded above. Under this
condition, the length of any other curve $\alpha$ can
be approximated up to  a bounded additive error by the
length of a polygonal curve. We describe here this
estimate, which is mainly obtained in \cite{DS}.

For notations
 we will say that two functions $f,g$ have the same
order, denoted by $f\sim g$, if there exists a constant $C>0$ so
that $\frac{1}{C}f \leq g \leq C f$. We will also use $f=O(1)$ to
indicate that  $f$ is bounded.

Let $(S,\sigma)$ be a hyperbolic surface and $\{\g_1,\dots,\g_N\}$
a pants decomposition system. When $\g_i,\g_j$  bound a pant,
denote  by $H_{ij}$ the simple common perpendicular arc between
them. For any pant $P$ and any boundary curve $\g_j$, we also
consider the arc $H_{jj}$ perpendicular to $\g_j$ and separating
the other two boundary components of $P$.

Now, any simple closed curve $\a$ can be homotoped to
a polygonal curve $BA_{\a}$ made up of arcs $V_j$
running along the geodesics $\g_j$, and arcs $H_{ij}$
($i$ may be equal to $j$).  This polygonal curve is
uniquely determined if we don't allow backtrackings;
i.e., if, for instance,
  $BA_{\a}$ goes along $H_{ij}$, then once around $\g_j$,
and then comes back along $H_{ij}$, then we change this part by
$H_{ii}$. (The notation $BA$ stands for the term ``broken arc" used
in \cite{DS}).

\begin{proposition}\label{brokenarc}
Let $\{\g_1,\dots \g_N\}$ be a pants decomposition of $S$, let
$M>0$, and let $\a$ be a simple closed curve of $S$. Then there
exists a constant $C$ (depending on $\a$ and $M$) such that for any
hyperbolic surface $(S,\sigma)$ with
 $\ell_{\sigma}(\g_j)<M$ for all $j=1,\dots, N$, we have
$$
 |\ell_\sigma(\a)- \sum_{j=1}^N i(\a,\gamma_j)\big[2\log
\frac{1}{\ell_\sigma(\gamma_j)}
+Tw_{\sigma}(\a,\gamma_j)\ell_\sigma(\gamma_j)\big]  |\leq C .
$$
\end{proposition}
\begin{proof}
Because the lengths of all the curves $\g_i$ are bounded above, the
orthogonal arcs between each two of them are greater than some
constant $D$, and  we can apply apply Lemma 5.1 in \cite{DS}. Then
we get that
\begin{equation}\label{BAequation}
|\ell_\sigma(\a)-  \ell_{\sigma}(BA_{a}) |\leq C' ,
\end{equation}
for some constant $C'$.

Next, we analyze  the lengths of the arcs $H_{ij}$ and $V_j$ in
$BA_{\a}$.  Since $\ell_{\s}(\g_j)$ are bounded above, it can be
shown by using the trigonometric formulae for pairs of pants that
$$
\ell_{\s}(H_{ij})=\log\frac{1}{\ell_{\s}(\g_i)}+
\log\frac{1}{\ell_{\s}(\g_j)}+O(1)
$$
(see Lemma 5.4 in \cite{DS}).

On the other hand, the length of the  each arc $V_j$ is close to
\break $Tw_{\sigma}(\a,\gamma_j)\ell_{\s}(\g_j)$. To see this,
observe that we can choose a marking $\mu$ associated to $\{\g_i\}$
so that, for any surface,
$$ |\frac{\ell_{\s}(V_j)}{\ell_{\s}(\g_j)}- t_j(\s) |<1 .$$
  Since $|t_j(\s)-tw_{\s}(\mu,\g_j)|<1$, we obtain that the
difference between $\frac{\ell_{\s}(V_j)}{\ell_{\s}(\g_j)}$ and
$Tw_{\s}(\mu,\g_j)$ is independent of the surface, up to bounded
error. Since this is also true for the difference between
$Tw_{\s}(\mu,\g_j)$ and $Tw_{\s}(\a,\g_j)$, we finally get that
$$   |Tw_{\s}(\a,\g_j)-\frac{\ell_{\s}(V_j)}{\ell_{\s}(\g_j)} |<O(1).$$
Plugging  the previous estimates in
(\ref{BAequation}), we finally obtain the result.

\end{proof}

An immediate consequence of the previous estimate that we will use
later is the following.

\begin{corollary}\label{easytwist}
Let $\{\g_1,\dots,\g_N;\mu\}$ be a marking for a
surface $S$, and $(S,\s_n)\in \teich$ a sequence such
that $\ell_{\s_n}(\g_j)$ is bounded above for all $j
$. Suppose further that, for some $k\in\{1,\dots,N\}$,
$\ell_{\s_n}(\g_k)$ is bounded below away from zero
and that there is a curve $\beta$  intersecting $\g_k
$ with $\ell_{\s_n}(\beta)$   bounded above. Then, the
twisting number $tw_{\s_n}(\beta,\g_k)$ and the twist
parameter $t_k(\s_n)$ are bounded.
\end{corollary}

{\bf Remark.} For arbitrary  surfaces (not satisfying the
condition explained above) the length of a simple closed curve can
also be estimated in terms of a ``thin-thick" decomposition of the
surface. For the length of the curve in the thin parts, the
estimates are the same as the one given above, but now it appear
new terms corresponding to the length of the arcs crossing the
thick parts. See \cite{cr} or \cite{crs} for detailed
explanations.

\subsection{Product region theorem} Let ${\cal A}=\{\g_1,\dots,\g_k\}$ be a collection
of disjoint, homotopically distinct, simple closed curves on $S$.
Given $\epsilon_0>0$ let $\cal T_{thin}(\cal A,\epsilon_0)\subset
\cal T(S)$ be the subset on which all curves in $\cal A$ have
length at most $\epsilon_0$. We extend $\cal A$ to a pants
decomposition system $\{\g_1,\dots,\g_k,$ $\g_{k+1},\dots,\g_N\}$,
choose a marking $\mu$ and consider Fenchel-Nielsen coordinates
with respect to this marking. Let $S_{\cal A}$ denote the surface
(possibly disconnected) obtained from $S$ by pinching all the
curves in $\cal A$. The marking chosen on $S$ induces  a marking
$\mu_{\cal A}$ on  $S_{\cal A}$, and we take Fenchel-Nielsen
coordinates on ${\cal T}(S_{\cal A})$ with respect to this
marking. Then we can define the map $\Pi_0\colon {\cal T}(S)\to
{\cal T}(S_{\cal A})$ by associating to a surface $(S,\sigma)$ the
surface in ${\cal T}(S_{\cal A})$ with Fenchel-Nielsen coordinates
$(\ell_\sigma(\g_{k+1}), \dots,\ell_\sigma(\g_{N}),$
$t_{\g_{k+1}}(\sigma), \dots,t_{\g_{N}}(\sigma))$ with respect to
the $\mu_{\cal A}$ (that is, we just
  forget  the coordinates corresponding to the curves
   in $\cal A$).   In \cite{Min}, the following
map is considered
$$\Pi:\cal T(S) \ra \cal T(S_{\cal A})\times H_{\g_1}\times
\cdots \times H_{\g_k},$$ where  $H_{\g_j}$ is a copy
of  the upper half plane, the first component of $\Pi$
is $\Pi_0$, and the remaining components
  $\Pi_{\g_j}:\cal T(S)\ra H_{\g_j}$ are
defined by
$$\Pi_{\g_j}(\sigma)=t_{\g_j}(\sigma)+i \frac{1}{\ell_\sigma(\g_j)}.$$
Let $d_{H_{\g_j}}$ be half the usual hyperbolic metric on
$H_{\g_j}$. Minsky proved \cite{Min} that the Teichm\"uller metric
on the thin parts can be approximated by the sup metric on the
previous product space. More precisely:
\begin{theorem} Given  $\epsilon_0$ sufficiently small, then for
any metrics $\sigma,\tau \in \cal T_{thin}(\cal
A,\epsilon_0)$, we have
$$d_{\cal T(S)}(\sigma,\tau)=\max_{\g\in \cal
A}\{d_{\cal T(S_{\cal
A})}(\Pi_0(\sigma),\Pi_0(\tau)),d_{H_\g}(\Pi_\g(\sigma),
\Pi_\g(\tau))\}\pm O(1).$$
\end{theorem}

A  consequence   of the theorem (see \cite{crs}) that we will use
later  is:  if $\sigma_1,\sigma_2\in\cal T_{thin}(\g,\epsilon_0)$
and for some $\nu\in \cal{ML}(S)$ with
$Tw_{\sigma_i}(\nu,\g)\ell_{\sigma_i}(\g)=O(1)$, then
\begin{equation}\label{Teichmullerdistance}
d_{H_\g}(\Pi_\g(\sigma_1),\Pi_\g(\sigma_2))=|\log
\frac{\ell_{\sigma_1}(\g)}{\ell_{\sigma_2}(\g)}|\pm O(1).
\end{equation}

\subsection{Grafting}\label{grafting} We will describe a map $Gr\colon\cal{
ML }(S)\times\teich \to P(S) $ and its composition
with the projection to $\teich$, $gr\colon\cal{ ML
}(S)\times \teich\to \teich$. For a simple closed
geodesic $\gamma$ on a hyperbolic surface $X$,
$gr_{t\gamma}(X)$ is constructed by cutting $X$ along
$\gamma$ and inserting a Euclidean right cylinder
$A(t)$ of height $t$ and circumference
$\ell_X(\gamma)$, with no twist. The Euclidean and
hyperbolic metric piece together continuously to give
a well-defined conformal structure. To define the {\it
projective} surface $Gr_{\g}(X)$, we first define a
projective model $A(t)$ of the inserted cylinder as
the quotient of the sector $\tilde A(t)=\{z\in \bc^* :
{\rm Arg }(z)\in [\frac{\pi}{2},\frac{\pi}{2}+t]\}$ by
the group generated by $\tau(z)=e^{\ell_{X}(\g)}z$
(when $t\geq 2\pi$, $\tilde A(t)$ must be interpreted
as multi-sheeted).  Then we glue the projective
cylinder $A(t)$ with the fuchsian structure on
$X-\{\g\}$.
 Identifying the
universal cover of $X$ with $\H^2$ so that $i\R$ is a component of
the lift of  $ \gamma$, a local  model for grafting is obtained by
cutting along $i\R$, applying the map
$$ z \mapsto \left\{
\begin{array}{ccl}
e^{it}z, & {\rm if }&{\rm Arg}(z)>\pi/2 \\
z,& {\rm if }&{\rm Arg}(z)\leq \pi/2
\end{array}
  \right.$$
and inserting the sector $\tilde A(t)$. Notice that the metric
$|dz|/|z|$ on $\tilde A(t)$ makes $\tilde A(t)/\langle \tau \rangle$
into a Euclidean cylinder of height $t$ and circumference
$\ell_{X}(\g)$, so that $Gr_{t\g}(X)$ is conformally identical to
$gr_{t\g}(X)$.


Grafting can be extended to general measured laminations by
continuity, giving
$$gr:\mathcal{ML}(S)\times\mathcal{T}(S)\rightarrow
\mathcal{T}(S).$$

Given a  projective structure $M$ on a surface, we can associate
two metrics to it: the projective (or Thurston) metric, and the
Kobayahi metric which coincides with the  hyperbolic metric
compatible with the underlying complex structure. We recall their
definitions (see \cite{Ta}).

For a tangent vector $v$ at a point $x\in M$, the {\it projective}
length is defined as the infimum of  the hyperbolic length of
vectors $v'$ in $T\H^2$ such that there exists a projective map
$f\colon\H^2 \rightarrow M$ sending $v'$ to $v$. The {\it
hyperbolic}   length is defined in the same way, but allowing $f$
to be  any holomorphic map. It is clear from the definition that
the hyperbolic length is less than or equal to the projective
length. On the other hand, the projective metric on $Gr_{t\g}(X)$,
obtained from $X$ with  a Euclidean cylinder inserted  along $\g$,
is just the combination of the hyperbolic metric of $X$ and the
flat metric of the cylinder. Also, as  a consequence of the
definition of the Kobayashi metric, any holomorphic map between
hyperbolic surfaces is distance decreasing.


 Recall that the {\it modulus} of the annulus
$A_r=\{z\in \bc : r<|z|<1\}$ is ${\rm mod}(A)
=\frac{1}{2\pi}\log\frac{1}{r}$, and the length of its core
geodesic with respect to its hyperbolic structure  is
$\ell_{A_r}(\g)=\frac{\pi}{{\rm mod}(A_r)}$. Also, the modulus of
a
  Euclidean right cylinder $E$ of height $t$ and circumference $\ell$ is
  ${\rm mod}(E)=t/\ell$.

For a grafted surface $gr_{t\g}(X)$ along a closed curve $\g$, we
can always consider a conformal homeomorphism $g$ from $ A_r$ to the
inserted    Euclidean cylinder $E $  (where $r$ is   chosen so that
$ {\rm mod}(A_r)={\rm
 mod}(E)$). Since
 $g$ is conformal, it is distance decreasing
with respect to the hyperbolic metric on $A_r$ and the Kobayashi
metric on $Gr_{t\g}(X)$. Therefore, we have the following estimate
for the length of $\g$ on $gr_{t\g}(X)$
$$
\ell_{gr_{t\g}(X)}(\g)\leq \ell_{A_r}(\g) =\frac{\pi}{ t}\ell_X(\g)
.
$$
We will obtain a slightly better estimate in Proposition
\ref{lengthbound}.

\medskip

Given a measured lamination $\lambda$ and a point
$X\in \teich$, we define the {\it grafting ray}   with
base $X$ determined by $\lambda$ as the set
$${\cal R}_{\lambda}(X)=\{gr_{t\lambda}(X) \;|\; t\geq 0 \},$$
and denote by ${\cal R_t}$ the surface
$gr_{t\lambda}(X)$.

\section{grafting and Thurston boundary of Teichm\"uller
space}

We  begin with some  estimates of the length of measured
laminations along grafting rays.
\begin{proposition}[McMullen \cite{mcm}]\label{ctm} For any
$\alpha$, $\gamma\in \mathcal{ML}(S)$ and
$X\in\mathcal{T}(S)$, we have
$$\ell_{gr_\g(X)}(\alpha)\leq\ell_X(\alpha)+i(\alpha,\g)$$
The inequality is strict if both $\alpha$ and $\g$ are nonzero.
\end{proposition}
\begin{proof}
When $\gamma$ is a  closed curve, then the projective length of
$\alpha$ at $Gr_{t\g}(X)$ is $\ell_X(\alpha)+ti(\a,\g)$, and the
result follows because the hyperbolic length on $Gr_{t\g}(X)$ is
smaller than the projective length. For general laminations, the
result follows by continuity.
\end{proof}
\begin{corollary}\label{cor1}
Consider the grafting ray $\cal R_{\g}(X)$. We have:
\begin{itemize}
\item[(a)] if $\alpha, \g\in\cal{ML}$ and $i(\a,\g)=0$, then
$\ell_{gr_{t\g}(X)}(\alpha)<\ell_X(\alpha)$ (so, the length of $\a$
is bounded above along the grafting ray);
\item[(b)] if $\a,\g$ are disjoint simple closed curves,
 then there exists a
 constant $C>0$ such that
$\ell_{gr_{t\g}(X)}(\alpha)>C$, for all $t>0$.
\end{itemize}
\end{corollary}
\begin{proof} Part (a) is direct application of  the previous
 proposition. For (b), first
 observe that,
 by (a), $\ell_{gr_{t\g}(X)}(\a)$ is bounded above.  Since
  $\alpha,\g$ are
disjoint, then there exists a measured lamination $\beta$
intersecting $\alpha$ but disjoint from $\g$; by (a),
$\ell_{gr_{t\g}(X)}(\beta) $ is bounded above. Suppose now that
$\ell_{gr_{t\g}(X)}(\a)\to 0$; then, since $\beta$ intersects
$\a$, it should be $\ell_{gr_{t\g}(X)}(\b)\to \infty$, arriving to
a contradiction.
\end{proof}

%
%
%
%

From this corollary we immediately find that, for $\lambda$ maximal
uniquely ergodic lamination,  the grafting ray   $\cal
R_{\lambda}(X) $ converges to $[\lambda]$. The same thing holds for
Teichm\"uller geodesics and lines of minima.   A measured lamination
is {\it maximal} if its support is not properly contained in the
support of any other measured lamination. A measured lamination is
{\it uniquely ergodic} if every measured lamination with the same
support is in the same projective class.

\begin{theorem}\label{con} Let $\lambda$ be a maximal uniquely
 ergodic measured lamination, and  $X\in \teich$. Then the
 grafting ray ${\cal R}_{\lambda}(X)$  converges to
$[\lambda]\in\mathcal{PML}(S)$ in Thurston's
compactification of $\mathcal{T}(S)$.
\end{theorem}

\begin{proof}

Since $\tc$ is compact, the grafting ray  has a
convergent subsequence. We will show that   the limits
of
 all the convergent subsequences of the grafting ray
  are equal to $[\lambda]$.

So, suppose that  $gr_{t_n\lambda}(X)\to[\nu]$. If
$|\nu|\not=|\lambda|$, since $\lambda $ is maximal,
then $i(\nu,\lambda)\not=0$ and  hence
$\ell_{gr_{t_n\lambda}(X)}(\lambda)$ tends to
infinity. Since this contradicts Corollary
\ref{cor1}(a) for $\a=\g=\lambda$, we must have that
$|\nu|=|\lambda|$. Since $\lambda $ is uniquely
ergodic, this implies that $[\nu]=[\lambda]$.
\end{proof}

Next, we are going to  obtain a better estimate of the lengths of
the grafting curves, when we do grafting along a simple closed
curve or a system of disjoint closed curves $\gamma_i$. In
particular, we will see that  along the grafting ray determined by
$\sum c_i\g_i$, the length of  each $\g_i$ tends to zero in the
order of $1/t$.

\begin{proposition}\label{lengthbound}
Let $\{\gamma_1,\cdots,\gamma_k\}$ be a system of
disjoint simple closed curves and let
$\lambda=\sum_jc_j\g_j$, with $c_j> 0$. Consider
$X\in\teich$. Then, for each $j=1,\dots, k$ we have
$$\frac{2\theta_0}{2\theta_0+ c t}\ell_X(\gamma_j)\leq
\ell_{gr_{t\lambda}(X)}(\gamma_j) \leq\frac{\pi}{\pi+ c_jt}
\ell_X(\gamma_j),$$ where $c=\max(c_1,\cdots,c_k)$ and $\theta_0$ is
some positive fixed number  depending only on $X$ (smaller than
$\pi/2$).


\end{proposition}
\begin{proof}
%
%

To prove  the right side
 inequality, assume first  that $\lambda$ is equal to a single simple
 closed curve $\g$. To compare the length of $\gamma$ in the surfaces
 $X$ and $X_t=gr_{t\g}(X)$, we look at the universal covers $\tilde
 X, \tilde X_t$ of these surfaces. The surface $\tilde X_t$ is the
 underlying conformal structure of the universal cover of ${Gr_{t\g}(X)}$, and this
 projective surface is built from $\H^2$ by inserting ``sectors"
 $\tilde A_j$ at any component $\tilde \g_j$ of the lift  of $\g$
 (when $\tilde \g_j$ is not a vertical line, $\tilde A_j$ is a region
 bounded by $\tilde \g_j$ and an equidistant line from $\tilde \g_j$, giving the effect that
 we get a bubble, see Figure 1). If $t$ is large, the universal cover $\tilde Gr_{t\g}(X)$ should be understood as
 multi-sheeted. For both universal covers, the map $z\mapsto
 e^{\ell}z$ is a covering transformation, where we are using the
 notation $\ell=\ell_X(\g)$.
      \begin{figure}[t]
    \vspace*{-3cm}
    \hspace*{-4.5cm}
    \psfig{file=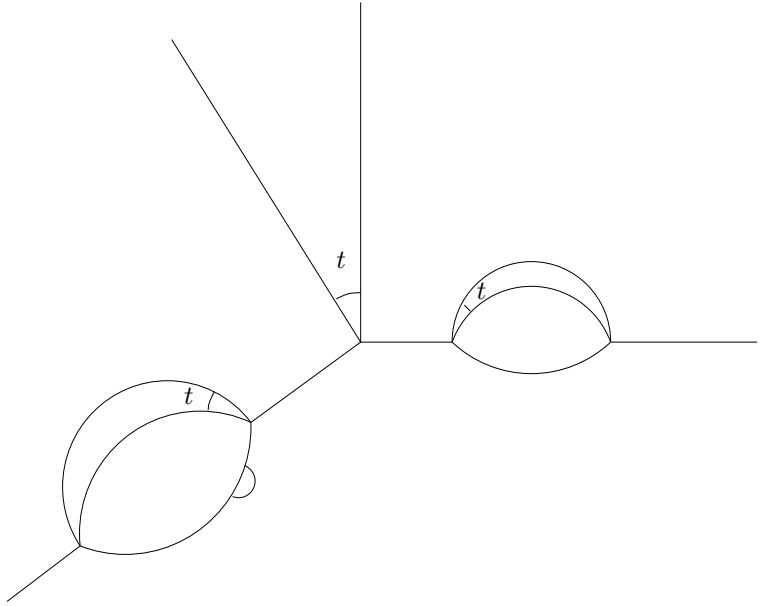}
    \vspace{-7in}
\caption{ Universal cover of  $Gr_{t\gamma}(X)$ when $t$ is small}
\end{figure}


 Consider the holomorphic map $\varphi(z)= z^{{\frac{\pi+t}{\pi}}}$
 from $\H^2$ to $\tilde Gr_{t\g}(X)$. Let $L$ be the vertical segment
 of $\H^2$ with endpoints $i, e^{\frac{\ell\pi}{\pi+t}}i$. Its image
 $\varphi(L)$ projects onto a closed curve homotopic to $\g$ in
 $X_t$. Therefore,
 $$
 \ell_{X_t}(\g)\leq \ell_{\tilde X_t}(\varphi(L))\leq
 \ell_{\H^2}(L)={\frac{\ell\pi}{\pi+t}}={\frac{\pi}
 {\pi+t}}\ell_X(\g),
 $$
  where the second inequality comes from the definition
  of the Kobayashi metric on $Gr_{t\g}(X)$. Then, we
  have the desired inequality.
 For the general case,  $\lambda=\sum_jc_j\g_j$, note that doing more
 graftings along the other curves means that one has to insert more
 sectors to build $\tilde Gr_{t\g}(X)$, and then the hyperbolic
 length on this domain is smaller than the one obtained in the previous
 computation by the definition of Kobayashi metric.


For the left side  inequality, we use the following general formula
that compares the  length of a closed curve in two different
hyperbolic surfaces in terms of the Teichm\"uller distance of these
surfaces (see \cite{Wolpert}). Let $X,X'\in \teich$ with
Teichm\"uller distance less than $C>0$. Then, for any closed curve
$\g$, we have that
\begin{equation}\label{otal}
 e^{-2C}\ell_{X'}(\g)\leq
\ell_X(\g)\leq e^{2C}\ell_{X'}(\g).
\end{equation}

To apply this formula in our case, we construct the following
quasiconformal map $\phi\colon X\to X_t=gr_{t\lambda}(X)$. Take
$\epsilon_0>0$ sufficiently small so that the
$\epsilon_0$-neighborhood of $\g_i$ is embedded on $X$. Denote by
$N_i$ this $\epsilon_0$-neighborhood, which is made up of two
annuli $A_{i,1},A_{i,2}$ of height $\epsilon_0$ glued at $\g_i$.
Take similar annuli $N_{i,t}$ in $X_t$, namely, $N_{i,t}$ is
obtained from  a  Euclidean cylinder of height $tc_i$ with
circumference $\ell_X(\gamma_i)$ by attaching $A_{i,1}$ and
$A_{i,2}$ at each side. Fundamental domains of $N_i,N_{i,t}$ are
conformally equivalent to the sets
$$\tilde N_i=\{re^{i\theta}\in \bc : r\in
[1,e^{\ell_{X}(\g)}],\theta\in[-\theta_0,\theta_0]\}$$
$$
\tilde N_{i,t}=\{re^{i\theta}\in \bc : r\in
[1,e^{\ell_{X}(\g)}],\theta\in[-\theta_0- c_i t,\theta_0]\},
$$
where $\theta_0$ is a concrete value in terms of $\epsilon_0$
 and,
for large $t$, $\tilde N_{i,t}$ should be understood
as multi-sheeted. Then the map $r\,{\rm
exp}(i\theta)\mapsto $ \break $  r\,{\rm exp}
(i(\frac{2\theta_0+ c_i t}{2\theta_0}\,\theta -\frac{
c_i t}{2})) $ is a quasiconformal homeomorphism from $
\tilde N_i$ to $\tilde N_{i,t}$ with maximal
dilatation $\frac{2\theta_0+c_it}{2\theta_0}$, and
that glues well to give a quasiconformal homeomorphism
$\phi_i$ between the quotient annuli. Extending the
homeomorphisms $\phi_i$ to $X-(\cup_i N_i)$ by the
identity map, we obtain a quasiconformal homeomorphism
$\phi$ with maximal dilatation
$\frac{2\theta_0+ct}{2\theta_0}$, where $c={\rm
max}\{c_1,\dots, c_k\}$.
Therefore the Teichm\"uller distance   between $X, X_t$ is smaller
than ${\frac{1}{2}}{\rm log} {\frac{2\theta_0+ct}{2\theta_0}}$, and
by using formula (\ref{otal}), we get the desired inequality.
\end{proof}

{\bf Remark 1}. In \cite{mcm}, McMullen sketched the proof of the
right side inequality of the above proposition by comparing the
moduli of the respective covers of $X$ and $gr_{t\g}(X)$
corresponding   to the subgroup of the fundamental group generated
by $\g$.

{\bf 2}. For the left side inequality, one can use some
inequalites derived in \cite{Ta}, but not as sharp as in this
proposition. When $\lambda$ is a weighted simple closed curve, an
asymptotic length estimate is given in \cite{dw}.

\medskip
Next, we bound  the twisting numbers of some curves
around the  curves $\g_j$ along which we do grafting.
We will do it for some special curves associated to a
pants decomposition, called dual curves.
  If
$\g_1,\dots, \g_N$ is a pants decomposition, a curve
$\delta_j$ is {\it dual} to $\g_j$ if
$i(\delta_j,\g_h)=0$ for all $h\not=j$ and
$i(\delta_j,\g_j)$ is $1$ or $2$ depending on whether
the  two pants of the decomposition  glued along
$\g_j$ are equal or different, respectively.

\begin{proposition}\label{twists} Let
$\{\g_1,\dots,\g_k\}$ be a set of disjoint simple closed curves and
$\lambda=c_1\g_1+\dots+c_k\g_k$. We add curves $\g_{k+1},\dots,\g_N$
to obtain a pants decomposition system, and let $\delta_j$ be dual
curves to  $\g_j$. Then, for each $j=1,\dots,N$ we have that $
Tw_{gr_{t\lambda}(X)}(\delta_j,\g_j) \ell_{gr_{t\lambda}(X)}(\g_j)$
is bounded above as $t\to \infty$.
\end{proposition}
\begin{proof}
If $j=k+1,\dots, N$, the result follows from Corollary
\ref{easytwist}, since curves not intersecting $\g_1,\dots,\g_k$
have length bounded above and below away from zero along the
grafting ray by Corollary \ref{cor1}.

Consider then $j=1,\dots,k$, and assume  that $i(\delta_j,\g_j)=2$,
that is, the two pairs of pants $P_1,P_2$ glued along $\g_j$ are
different (the other case is similar).  Let $X_t=gr_{t\lambda}(X)$.
Since $\ell_{X_t}(\g_j) $ is bounded above for all $j$, we can apply
Proposition \ref{brokenarc}; then, there exists a constant $C$
independent of $t$ such that
\begin{equation} \label{bigger}
\ell_{X_t}(\delta_j)>4 \log \frac{1}{\ell_{X_t}(\g_j)}+ 2
Tw_{X_t}(\delta_j,\g_j) \ell_{X_t}(\g_j) -C
\end{equation}


 To estimate the length of $\delta_j$
from above, recall that  $X_t$ is a uniformization of the
projective surface constructed by cutting $X$ along $\g_j$ and
inserting a Euclidean cylinder $E_j$; then the geodesic $\delta_j$
of $X$ is splitted into two arcs $a_1,a_2$; we join these arcs by
two horizontal arcs $L_t,L'_t$ of the inserted cylinder, to get
the curve $a_1\cup L_t\cup a_2\cup L'_t$, that represents
$\delta_j$ in $X_t$ (see Figure 2). Then, the hyperbolic length of
the geodesic representing $\delta_j$ satisfies
$$
\ell_{X_t}(\delta_j) \leq
\ell_{X_t}(a_1)+\ell_{X_t}(a_2)+\ell_{X_t}(L_t)+\ell_{X_t}(L'_t).
$$
      \begin{figure}[t]
    \vspace*{-3cm}
    \hspace*{-4.5cm}
    \psfig{file=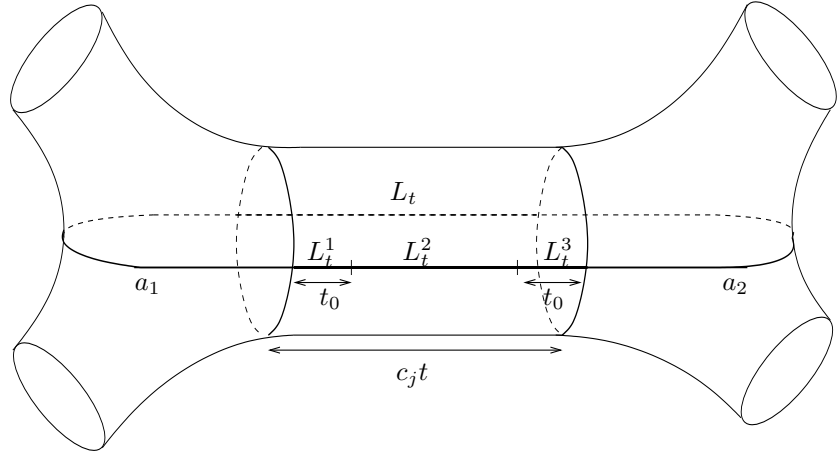}
    \vspace{-7.7in}
\caption{Computing twists }
\end{figure}

Now, $\ell_{X_t}(a_i)<\ell_X(a_i)$ because the hyperbolic metric in
$X_t$ is smaller than the projective metric, and the latter
coincides with the hyperbolic metric of $X$ on the arcs $a_1,a_2$.

It is left to estimate $\ell_{X_t}(L_t)$ when  $t\to \infty$. In
order to do this, take $t_0$ some fixed   positive number
(independent of $t$) and divide the arc $L_t$ as union of three
segments $L_t^1\cup L_t^2\cup L_t^3$, where both $L_t^1$ and $L_t^3$
have Euclidean length $t_0$. Since the  Kobayashi metric is smaller
than the projective metric, and the latter  coincides with  the
Euclidean metric on the inserted annuli,  we have that
$\ell_{X_t}(L_t^i)<t_0$, for $i=1,3$.

We finally estimate $\ell_{X_t}(L_t^2)$.    As in
section \ref{grafting}, consider a holomorphic
embedding $g\colon A_r\to E_j$, where $r$ is such that
${\rm mod}(A_r)={\rm mod}(E_j)$; thus $r=e^{-\frac{2\pi c_jt}{\ell}}$, where
$\ell=\ell_X(\g_j)$. Taking $E_j$ as the
rectangle $[0,c_jt]\times [0,\ell]$
 with the horizontal sides being identified, then the map $g$ has inverse
  $g^{-1}(z)={\rm exp}
  (\frac{-2\pi}{\ell }z)$. On the other hand, the uniformizing map for $A_r$ is
  $\Phi\colon \H^2\to A_r$ defined as $\Phi(z)=z^{-\frac{{\rm log} r}{\pi}i}$.
  Then, we can see that $g^{-1}(L_t^2)$ is a geodesic arc on the
  hyperbolic metric of $A_r$, whose length is equal to the hyperbolic length of one
   of the connected components of  $\Phi^{-1}(g^{-1}(L^2_t))$
   (see Figure 3).

   \begin{figure}[t]
    \vspace*{-3cm}
    \hspace*{-7cm}
    \psfig{file=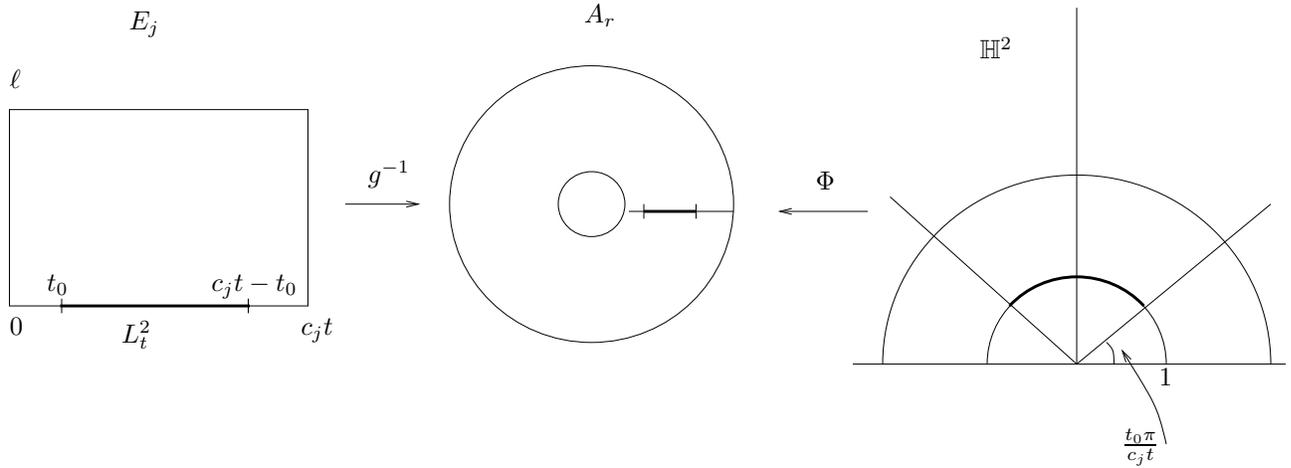}
    \vspace{-7.7in}
\caption{Computing twists (ii) }
\end{figure}

   We compute this length by elementary hyperbolic geometry and, using that $g$
   is distance decreasing, we get
  $$
\ell_{X_t}(L_t^2)\leq \ell_{A_r}(g^{-1}(L_t^2) )=
 2\log\frac{{\rm cos}\,
\frac{t_0\pi}{2c_jt}}{{\rm sin}\, \frac{t_0\pi}{2c_jt}} .
  $$
Finally, notice that
$$\frac{{\rm cos}\, \frac{t_0\pi}{2c_jt}}{{\rm sin}\,
\frac{t_0\pi}{2c_jt}}\sim t \; {\rm as} \;t\to \infty
,$$
so that we have that
  $\ell_{X_t}(L_t^2)\leq  2\log (t) +O(1)$.
   The same estimate holds for $L'_t$. Thus, putting all these together,
   we  have
  \begin{equation}\label{smaller}
\ell_{X_t}(\delta_j) \leq
 4\log(t)+O(1).
 \end{equation}

 Combining (\ref{bigger}) and (\ref{smaller}), taking into account that
 $\frac{1}{\ell_{X_t}(\g_j)}\sim t$ (by Proposition \ref{lengthbound}),
   and
 cancelling out the terms
 $\log (t)$, we get that the term
 $Tw_{X_t}(\delta_j,\g_j)\ell_{X_t}(\g_j)$ is
 bounded as $t\to \infty$.
\end{proof}


\bigskip

Using the previous results we can now study the convergence  to
the Thurston's boundary of the grafting rays determined   by
rational laminations.

\begin{theorem}
\label{limit} Let $\{\gamma_1,\cdots,\gamma_k\}$ be a
system of disjoint simple closed  curves, let
  $\lambda=\sum c_i\gamma_i$ be
a weighted system, and consider any $X\in\teich$.
Then
$$gr_{t\lambda}(X)\ra [\sum\gamma_i]$$ in Thurston boundary of
Teichm\"uller space.
\end{theorem}
\begin{proof}
 We need to show that for any two simple closed curves
 $\beta,\beta'$,
 $$\frac{\ell_{X_t}(\beta)}{\ell_{X_t}(\beta')}\ra
 \frac{i(\beta,\sum\gamma_i)}{i(\beta',\sum\gamma_i)},$$
 where $X_t=gr_{t\lambda}(X).$

 Extend $\{\g_1,\dots,\g_k\} $ to a pants decomposition
 $\{\g_1,\dots,\g_N\}$. By Corollary \ref{cor1} all the
 $\g_j$ have length bounded above along the grafting
 ray. Therefore,   applying Proposition
 \ref{brokenarc}, we have
 $$
 \ell_{X_t}(\beta)=\sum_{j=1}^Ni(\beta,\g_j)
 [2\log\frac{1}{\ell_{X_t}(\g_j)}+
 Tw_{X_t}(\beta,\g_j)\ell_{X_t}(\g_j) ]+O(1).
 $$

By Proposition \ref{twists}, all the terms
$Tw_{X_t}(\beta,\g_j)\ell_{X_t}(\g_j) $ are bounded above.
 Since $\ell_{X_t}(\g_j)$ is bounded below away from
 zero for $j=k+1,\dots,N$ (Corollary \ref{cor1}),   the
 terms of the previous summation  from $j=k+1$ to $j=N$ are
 bounded above.
 Therefore we get
$$
 \ell_{X_t}(\beta)=2\sum_{j=1}^ki(\beta,\g_j)
 \log\frac{1}{\ell_{X_t}(\g_j)} +O(1),
 $$
and similarly for $\beta'$.
  By Proposition \ref{lengthbound}, we have
that $\ell_{X_t}(\g_i)\sim 1/t$, for all $i=1,\dots ,
k$, so that
$$\frac{\log \frac{1}{\ell_{X_t}(\g_i)}} {\log(t)}\to 1.
$$

Therefore, to compute the limit of  the quotient
$\frac{\ell_{X_t}(\beta_i)} {\ell_{X_t}(\beta_j)}$, we
first divide numerator and denominator by $\log(t)$,
then all the bounded terms vanish  and we get the
desired result.
\end{proof}

\section{Grafting ray is a Quasi-geodesic in Teichm\"uller
space}

In this section we will use Minsky's product region theorem to
estimate the distance between a grafting ray ${\cal
R}_{\lambda}(X)$ based on a surface $X\in \teich$ and  the
Teichm\"uller geodesic ray ${\cal G}(\lambda,X)$  based on the
same surface and with vertical foliation $\lambda$, when
$\lambda=c_1\g_1+\dots+c_k\g_k$.

 We first give some   results needed to analyze the different
 components of Minsky's formula.



\begin{proposition}\label{picero} Consider the
situation described in section 2.5
 to define the map $\Pi_0\colon \teich \to
 {\cal T}(S_{\cal A})$;  let
$X_t\in \teich$  be a family such that the length of any closed
curve $\beta$ disjoint from the curves in $\cal A$ is bounded
above. Then, the family $\Pi_0(X_t)$ is contained in a compact
subset of $ {\cal T}(S_{\cal A})$.
\end{proposition}

\begin{proof}
Let $(\ell_{X_t}(\g_j), t_j(X_t))_{j=1,\dots,N}$ be
the Fenchel-Nielsen coordinates of $X_t$ with respect
to the marking $\{\g_j;\mu\}$. Because the lengths of
all simple closed curves disjoint from $\g_1,\dots,
\g_k$ are bounded above, then these lengths are also
bounded below away from zero, and the twist parameters
$t_{k+1}(X_t), \dots,t_{N}(X_t)$ are all bounded by
Corollary \ref{easytwist}. Hence, we have that
$(\ell_{X_t}(\g_j), t_j(X_t))_{j=k+1,\dots,N}$ is
contained in a compact subset of $(\R_+\times
\R)^{2N-2k} $, and this last set parameterizes ${\cal
T}(S_{\cal A})$.
\end{proof}

Next, we give some properties of  Teichm\"uller
geodesics, mainly extracted from  \cite{Masur} (see
also \cite{crs}).

\begin{proposition}\label{masur}
Let $\{\g_1,\dots,\g_k\}$ be a system of disjoint
simple closed curves  and
$\lambda=c_1\g_1+\dots+c_k\g_k$, $c_i>0$. Let $X\in
\teich$ and consider $\varphi\in \Sigma^1(X)$ with
vertical foliation projectively equivalent to
$\lambda$. Let ${\cal G}(\lambda,X)$ be the
Teichm\"uller geodesic ray determined by $\varphi$,
and ${\cal G}_t={\cal
G}_t(\lambda,X)=(X,\lambda,\frac{t}{t+2})$. Then we
have:
\begin{itemize}
\item[i)] the length of any simple closed geodesic disjoint from
$\g_i$ ($i=1,\dots, k$) is bounded above. \item[ii)] for all
$i=1,\dots,k$ we have $\ell_{{\cal G}_t} (\g_i)\sim
\frac{1}{t+1}\,$ for $t\to \infty$; \item[iii)] for any $\alpha$
intersecting $\g_i$, $Tw_{{\cal G}_t} (\alpha,\g_i)\ell_{{\cal
G}_t}(\g_i)$ is bounded.

\end{itemize}
\end{proposition}

\begin{proof}
i) By  \cite{Masur} Lemma 3 (v), the hyperbolic metrics  on the
points of the Teichm\"uller geodesic ray converge to a  specific
hyperbolic metric on the punctured surface $S_{\cal A}$.
Therefore, any curve disjoint from $\g_i$ converges to a curve in
the limiting surface, and hence  its length is bounded above.

 ii)
A quadratic differential like the one in the
statement, with all the leaves of the vertical
foliation being closed curves, is called a {\it
Jenkins-Strebel differential}. The leaves of this
foliation are grouped into conformal annuli $A_i$
($i=1,\dots,k$) with core curve homotopic to $\g_i$.
Let $M_i$ be the moduli of these annuli (which depends
on the coefficient $c_i$ and the hyperbolic length of
$\g_i$) on $X$. The surface ${\cal
G}_t=(X,\lambda,\frac{t}{t+2})$ is also decomposed
into conformal annuli $A_{i,t}$, whose moduli are now
$(t+1)M_i$. The hyperbolic length of the core geodesic
of $A_{i,t}$ is $\frac{\pi}{(t+1)M_i} $. Now, the
result follows from Lemma 4 in \cite{Masur}, where the
author proves that the hyperbolic metric on ${\cal
G}_t$ and the hyperbolic metric on $A_{i,t}$ have the
same order (as $t\to \infty$) at all  the points of a
certain subannulus $A_{\delta}$ of $A_{i,t}$, which
contains the geodesic core of $A_{i,t}$.

iii) The proof of Proposition \ref{twists} works also in the
situation explained in \cite{Masur} for Teichm\"uller geodesic
rays.

%
\end{proof}

We finally prove the theorem.

\begin{theorem}\label{quasigeodesic} Let
$\{\gamma_1,\cdots,\gamma_k\}$ be a system of disjoint
simple closed  curves on $S$, let
  $\lambda=\sum c_i\gamma_i$ be
a weighted system, and consider any $X\in\teich$.
Then, the grafting ray
 ${\cal R}_{\lambda}(X)$ is at
bounded distance from the Teichm\"uller geodesic ray
${\cal G} (\lambda, X)$.
\end{theorem}
\begin{proof}
We will show that $d_{\teich}({\cal R}_t,{\cal G}_{t} )$ is
bounded above. As in section 2.5, extend the $\g_i$ to a pants
decomposition and take a marking. By Corollary \ref{cor1} and
Proposition \ref{masur}(i), the length of curves disjoint from
$\g_1,\dots,\g_k$ have bounded length both along the grafting and
the Teichm\"uller rays. Then, by Proposition \ref{picero}, $\{
\Pi_0({\cal R}_t),$ $\Pi_0({\cal G}_{t}) : t\geq 0 \}$ is
contained in a compact set of ${\cal T}(S_{\cal A})$, and
therefore \break $d_{{\cal T}(S_{\cal A})}(\Pi_0({\cal
R}_t),\Pi_0({\cal G}_{t} ))$ is bounded as $t\to \infty$.


On the other hand, by Propositions \ref{lengthbound}
and \ref{twists}, we have
$$\ell_{\cal R_t}(\gamma_i)\sim 1/t,$$
$$Tw_{\cal R_t}(\delta_{\gamma_i},\gamma_i)
\ell_{\cal R_t}(\gamma_i)=O(1);$$
 and  by Proposition
\ref{masur},
$$
\ell_{\cal G_t}(\gamma_i)\sim \frac{1}{t+1},$$
$$Tw_{\cal G_t}(\delta_{\gamma_i},\gamma_i)\ell_{\cal
G_t}(\gamma_i)=O(1).$$
 Then, by equation
(\ref{Teichmullerdistance}), we have
$$d_{H_{\gamma_i}}(\Pi_{\gamma_i}(\cal R_t),
\Pi_{\gamma_i}(\cal G_{t}) )=|\log \frac{\ell_{\cal G_{t}}
(\gamma_i)}{\ell_{\cal R_t}(\gamma_i)}|\pm O(1).$$

Putting all these estimates into Minsky's product region theorem,
we finally obtain that the Teichm\"uller distance between $\cal
G_{t}$ and $\cal R_t$ is bounded.
\end{proof}

 {\bf ACKNOWLEDGEMENTS} The
authors greatly benefited from the conversations with
C. McMullen.

\bigskip
\vskip .05 in \noindent
Departamento de Geometr\'{\i}a y Topolog\'{\i}a\\
 Fac. CC. Matem\'aticas\\
  Universidad Complutense de Madrid, 28040 Madrid, Espa\~na\\e-mail:
 radiaz\char`\@
 mat.ucm.es
\vskip .09 in \noindent
Department of Mathematics\\
 Seoul National University, Seoul 151-742, Korea\\e-mail:
 inkang\char`\@
 math.snu.ac.kr

 \end{document}